\documentclass[a4paper,12pt]{article}
\usepackage{amsmath}
\usepackage{amsthm}
\usepackage{amssymb}
\usepackage{enumerate}
\usepackage{url}

\title{Some asymptotic formulae for Bessel process\thanks{
This work was partially supported by JSPS KAKENHI 
Grant Number 24740080.}}

\author{Yuu Hariya\thanks{Mathematical Institute, 
Tohoku University, Aoba-ku, Sendai 980-8578, Japan. }}

\date{\empty}
\numberwithin{equation}{section}

\theoremstyle{plain}

\newtheorem{thm}{Theorem}[section]

\newtheorem{prop}{Proposition}[section]

\newtheorem{lem}{Lemma}[section]

\theoremstyle{definition}

\theoremstyle{remark}
\newtheorem{rem}{Remark}[section]

\begin{document}

\def\N {\mathbb{N}}
\def\R {\mathbb{R}}
\def\Q {\mathbb{Q}}

\def\calF {\mathcal{F}}

\def\kp {\kappa}

\def\ind {\boldsymbol{1}}

\def\al {\alpha }
\def\a {\al }
\def\la {\lambda }
\def\ve {\varepsilon}
\def\Om {\Omega}

\def\ga {\gamma }

\def\W {W}

\newcommand\ND{\newcommand}
\newcommand\RD{\renewcommand}

\ND\lref[1]{Lemma~\ref{#1}}
\ND\tref[1]{Theorem~\ref{#1}}
\ND\pref[1]{Proposition~\ref{#1}}
\ND\sref[1]{Section~\ref{#1}}
\ND\ssref[1]{Subsection~\ref{#1}}
\ND\aref[1]{Appendix~\ref{#1}}
\ND\rref[1]{Remark~\ref{#1}} 
\ND\cref[1]{Corollary~\ref{#1}}
\ND\eref[1]{Example~\ref{#1}}
\ND\fref[1]{Fig.\ {#1} }
\ND\lsref[1]{Lemmas~\ref{#1}}
\ND\tsref[1]{Theorems~\ref{#1}}
\ND\dref[1]{Definition~\ref{#1}}
\ND\psref[1]{Propositions~\ref{#1}}

\ND\pr[2]{P^{(#1)}_{#2}}
\ND\ex[2]{E^{(#1)}_{#2}}

\ND\br[1]{B^{(#1)}}
\ND\be[1]{R^{(#1)}}

\def\ri {\rho _{\infty }}

\def\thefootnote{{}}

\maketitle 

\begin{abstract}
 We recover in part a recent result of \cite{hm} on the asymptotic behaviors 
 for tail probabilities of first hitting times of Bessel process. 
 Our proof is based on a weak convergence argument. The same reasoning 
 enables us to derive the asymptotic behaviors for the tail probability of 
 the time at which the global infimum of Bessel process is attained, and 
 for expected values relative to local infima. 
 In addition, we give another proof of the result of 
 \cite{hm} with improvement of error estimates, which complements 
 in the case of noninteger dimensions the asymptotic formulae by 
 \cite{vdb} for first hitting times of multidimensional Brownian motion. 
\footnote{E-mail: hariya@math.tohoku.ac.jp}
\end{abstract}

\section{Introduction}\label{;intro}
For every $\nu \in \R $ and $a>0$, we denote by $\pr{\nu }{a}$ the 
law on the space $C([0,\infty );\R )$ of real-valued continuous paths 
over $[0,\infty )$, induced by 
Bessel process with index $\nu $ (or dimension $\delta =2(\nu +1)$) 
starting from $a$. For every $b\ge 0$, we denote by $\tau _{b}$ 
the first hitting time to $b$: 
\begin{align*}
 \tau _{b}(\omega ):=\inf \{ 
 t\ge 0;\,\omega (t)=b
 \} , \quad \omega \in C([0,\infty );\R ). 
\end{align*}
In Hamana-Matsumoto \cite{hm}, they have shown the following asymptotic 
formulae for the tail probability of $\tau _{b}$ in the case $b<a$: 
for every $\nu >0$, 
\begin{align}
 \pr{\nu }{a}(\infty >\tau _{b}>t)
 &=\frac{1}{(2t)^{\nu }\Gamma (\nu +1)}
 b^{2\nu }\left\{ 
 1-\left( \frac{b}{a}\right) ^{2\nu }
 \right\} +O(t^{-\nu -\ve }), \label{;asym1}\\
 \pr{-\nu }{a}(\tau _{b}>t)
 &=\frac{1}{(2t)^{\nu }\Gamma (\nu +1)}
 a^{2\nu }\left\{ 
 1-\left( \frac{b}{a}\right) ^{2\nu }
 \right\} +O(t^{-\nu -\ve })
 \label{;asym2}
\end{align}
as $t\to \infty $ for any $\ve \in (0,\nu /(\nu +1))$. 
Here $\Gamma $ is the gamma function. 
Their proof uses computational estimates. 

One of the purposes of this paper is to give a different proof of these 
two formulae based on a weak convergence argument; 
while our proof does not 
give asymptotic estimates for remainder terms as in \eqref{;asym1} and 
\eqref{;asym2}, we think that it is straightforward. 
The same reasoning also provides the asymptotic behaviors for the tail 
probability of the time at which the global infimum of Bessel process is 
attained, and for some expected values related to local infima. 

We devote the latter half of the paper to another proof of \eqref{;asym1} 
and \eqref{;asym2}. The proof is based on an identity for hitting 
distributions that is an immediate consequence of the strong Markov 
property of Bessel process. The identity differs from the one used in 
\cite{hm} and makes it possible to do more precise estimates; 
it will be shown that the remainder terms decay at rate $t^{-2\nu }$ 
for $0<\nu <1$ except $\nu =1/2$, $(\log t)/t^{2}$ for $\nu =1$, 
and $t^{-\nu -1}$ for $\nu >1$ and $\nu =1/2$. See \tsref{;improve} 
and \ref{;improve2} and \rref{;rimprove} below. In the case the 
dimension $\delta $ is noninteger, these asymptotics complement a 
result by van den Berg \cite{vdb} that deals with first hitting times 
of Brownian motion to general compact sets in dimension greater than 
or equal to 3; repeated use of comparison of the first hitting time with 
the last exit time is the method employed there, which our argument is 
also different from. We also remark that we do not treat the case 
$\nu =0$, for which we refer the reader to the detailed study \cite{uch0} 
by Uchiyama, where obtained are asymptotics of hitting distributions in 
question as well as asymptotic estimates on density functions of first 
hitting times. 

We organize this paper as follows: 
In \sref{;results} we first prove \tref{;th1}, which recovers principal 
terms in \eqref{;asym1} and \eqref{;asym2}; we reduce the proof to 
showing that a given sequence of probability measures on $C([0,\infty );\R )$ 
is weakly convergent. This argument also proves \pref{;keyprop}, 
which is then applied to derive in \tref{;tcor} the asymptotic behavior 
for the tail probability of the time Bessel process with positive 
index attains its global infimum, and those for expected values 
involving its local infima. 
In \sref{;improved} we prove 
\tsref{;improve} and \ref{;improve2} that improve \eqref{;asym1} and 
\eqref{;asym2}; 
we do this by using an identity for hitting distributions 
given in \lref{;hitting}. Finally in the appendix, we prove 
auxiliary facts that are referred to in Sections~\ref{;results} and 
\ref{;improved}. 
\bigskip 

In the sequel we write $\Om $ for $C([0,\infty );\R )$. 
We equip $\Om $ with the topology of compact uniform convergence. 
Unless otherwise stated, $R$ denotes the coordinate 
process on $\Om $: $R_{t}(\omega ):=\omega (t),\,\omega \in \Om ,t\ge 0$. 
We also set 
\begin{align*}
 \calF _{t}:=\sigma (R_{s},s\le t),\ t\ge 0, 
 \quad \text{and}\quad \calF :=\bigvee _{t\ge 0}\calF _{t}. 
\end{align*}
For any $x,y\in \R $, we write $x\vee y=\max\{ x,y\} ,\ 
x\wedge y=\min \{ x,y\} $. Other notation will be introduced 
as needed. 

\section{Main results and proofs}\label{;results}
Throughout the paper $\nu $ denotes a positive index. 
One of the objectives of this section is to give a proof of 
\begin{thm}\label{;th1}
It holds that 
for every $b\in [0,a)$, 
\begin{align*}
 \thetag{i}&~\lim _{t\to \infty }t^{\nu }\pr{\nu }{a}
 (\infty >\tau _{b}>t)=
 \frac{1}{2^{\nu }\Gamma (\nu +1)}
 b^{2\nu }\left\{ 
 1-\left( \frac{b}{a}\right) ^{2\nu }
 \right\} ;\\
 \thetag{ii}&~
 \lim _{t\to \infty }t^{\nu }\pr{-\nu }{a}
 (\tau _{b}>t)=
 \frac{1}{2^{\nu }\Gamma (\nu +1)}
 a^{2\nu }\left\{ 
 1-\left( \frac{b}{a}\right) ^{2\nu }
 \right\} . 
\end{align*}
\end{thm}

We begin with stating two facts in propositions. 

\begin{prop}\label{;absrel}
 It holds that for every $t>0$, 
 \begin{align}\label{;opposite}
  \pr{\nu }{a}\big| _{\calF _{t}}
  =\left( \frac{R_{t}}{a}\right) ^{2\nu }
  \pr{-\nu }{a}\big| _{\calF _{t}\cap \{ t<\tau _{0}\} }. 
 \end{align}
\end{prop}

\begin{prop}\label{;taboo}
Let $t>0$. Then for any $A\in \calF _{t}$, we have 
\begin{align*}
 \lim _{s\to \infty }\pr{-\nu }{a}(A\,|\,\tau _{0}>s)
 =\pr{\nu }{a}(A). 
\end{align*}
\end{prop}

These two facts seem well known; to make the paper self-contained, we 
provide their proofs in Appendix. We deduce from \pref{;absrel} the 
following lemma, which plays a key role throughout this section. 

\begin{lem}\label{;keylem}
 For every 
 $x>0$, it holds that as $t\to \infty $, 
 \begin{align}\label{;keyasympt}
  t^{\nu }\ex{\nu }{x}\left[ 
  \left( \frac{1}{R_{t}}\right) ^{2\nu }
  \right] 
  \to \frac{1}{2^{\nu }\Gamma (\nu +1)}. 
 \end{align}
 Here and below, $\ex{\nu }{x}$ denotes the expectation 
 with respect to $\pr{\nu }{x}$. 
\end{lem}

\begin{proof}
 By the absolute continuity relation \pref{;absrel}, 
 the expectation on the left-hand side of \eqref{;keyasympt} 
 is equal to $x^{-2\nu }\pr{-\nu }{x}(\tau _{0}>t)$, and hence 
 admits the representation 
 \begin{align*}
  \ex{\nu }{x}\left[ 
  \left( \frac{1}{R_{t}}\right) ^{2\nu }
  \right] 
  =\frac{1}{2^{\nu }\Gamma (\nu )}\int _{t}^\infty 
  \frac{ds}{s^{\nu +1}}\exp \left( -\frac{x^2}{2s}\right) 
 \end{align*}
 (see \rref{;gamma} in Appendix). 
 The assertion readily follows from this identity. 
\end{proof}
One may also prove the lemma by using the explicit 
representation for the transition densities of Bessel process. 

For each $t\ge 0$, we set 
\begin{align*}
 I_{t}\equiv I_{t}(R):=\inf _{0\le s\le t}R_{s}. 
\end{align*}
We also write $I_{\infty }$ for $\inf \limits_{t\ge 0}R_{t}$. Recall that 
for every $x>0$ and $0\le y\le x$, 
\begin{align}\label{;idistr}
 \pr{\nu }{x}\left( I_{\infty }>y\right) 
 =1-\left( \frac{y}{x}\right) ^{2\nu }. 
\end{align}
As in \cite{hm}, we also use the following identity: 
\begin{lem}\label{;keyiden}
For every $b\in [0,a)$ and $t>0$, it holds that 
\begin{align}\label{;iden}
 \pr{\nu }{a}(I_{t}>b)
 =1-\left( \frac{b}{a}\right) ^{2\nu }
 +\ex{\nu }{a}\left[ 
 \left( \frac{b}{R_{t}}\right) ^{2\nu }; 
 I_{t}>b
 \right] . 
\end{align}
\end{lem}

\begin{proof}
 By the Markov property of Bessel process and \eqref{;idistr}, 
 \begin{align*}
  \pr{\nu }{a}(I_{\infty }>b\,|\,\calF _{t})&=
  \pr{\nu }{x}(y\wedge I_{\infty }>b)\big|_{(x,y)=(R_{t},I_{t})}\\
  &=\ind _{\{ I_{t}>b\} }
  \left\{ 1-\left( \frac{b}{R_{t}}\right) ^{2\nu }\right\} 
 \end{align*}
 $\pr{\nu }{a}$-a.s. 
 Taking the expectation on both sides leads to \eqref{;iden}. 
\end{proof}

Since $\pr{\nu }{a}(\tau _{b}>t)=\pr{\nu }{a}(I_{t}>b)$ and 
$\pr{\nu }{a}(\tau _{b}=\infty )=\pr{\nu }{a}(I_\infty >b)
=1-(b/a)^{2\nu }$ by \eqref{;idistr}, we have from \eqref{;iden} 
\begin{align}\label{;iden2}
 \pr{\nu }{a}(\infty >\tau _{b}>t)=
 \ex{\nu }{a}\left[ 
 \left( \frac{b}{R_{t}}\right) ^{2\nu }; 
 I_{t}>b
 \right] . 
\end{align}
We are ready to prove \tref{;th1}. 

\def\tpn {\tilde{P}_{n}}
\begin{proof}[Proof of \tref{;th1}]
\thetag{i} Fix arbitrarily a strictly increasing sequence 
$\{ t_{n}\} _{n\in \N }\subset (0,\infty )$ such that 
$\lim \limits_{n\to \infty }t_{n}=\infty $. For each $n$, 
we define the probability measure $\tpn $ on $\Om $ by 
\begin{align*}
 \tpn (A):=
 \frac{
 \ex{\nu }{a}\left[ (R_{t_{n}})^{-2\nu };R^{t_{n}}_{\cdot }\in A\right] 
 }{\ex{\nu }{a}[(R_{t_{n}})^{-2\nu }]}, 
 \quad A\in \calF , 
\end{align*}
where $R^{t_{n}}_{t}:=R_{t\wedge t_{n}},\,t\ge 0$. 

First we show that 
$\{ \tpn \} _{n\in \N }$ is tight. Fix $t>0$ and take $A'\in \calF _{t}$. 
If we let $n$ be such that $t_{n}\ge t$, then by \psref{;absrel} and 
\ref{;taboo}, 
\begin{align}
 \tpn (A')&=\pr{-\nu }{a}(A'\,|\,\tau _{0}>t_{n}) \notag \\
 &\xrightarrow[n\to \infty ]{}\pr{\nu }{a}(A'). \label{;preq1}
\end{align}
This convergence for any $A'\in \calF _{t}$ entails in particular that 
by regarding each $\tpn $ as being defined 
on the path space $\Om _{t}=C([0,t];\R )$ equipped with the uniform norm 
topology, $\{ \tpn \} _{n\in \N }$ is tight as 
a sequence of probability measures on $\Om _{t}$, which is equivalent to 
\begin{align}\label{;equicont}
 \lim _{\delta \downarrow 0}\sup _{n\in \N }
 \tpn \Bigl( \omega \in \Om _{t};\,
 \sup _{
 \begin{subarray}{c}
 |u-v|\le \delta \\
 0\le u,v\le t
 \end{subarray}
 }
 |\omega (u)-\omega (v)|>\ve 
 \Bigr) =0
\end{align}
for any $\ve >0$ (see, e.g., \cite[Theorem~2.4.10]{ks}). 
It is then clear that, with $\Om _{t}$ replaced by $\Om $, 
\eqref{;equicont} holds for any $t>0$ and $\ve >0$, 
and hence the tightness of 
$\{ \tpn \} _{n\in \N }$ follows. 

As $t>0$ is arbitrary, the convergence $\eqref{;preq1}$ 
also implies that $\{ \tpn \} _{n\in \N }$ converges to $\pr{\nu }{a}$ 
in the sense of finite-dimensional distributions. Consequently, 
$\{ \tpn \} _{n\in \N }$ converges weakly to $\pr{\nu }{a}$. 
Since $\pr{\nu }{a}(I_{\infty }=b)=0$ by \eqref{;idistr}, 
the weak convergence entails that 
\begin{align}\label{;conv}
 \tpn (I_{\infty }>b)\xrightarrow[n\to \infty ]{}
 \pr{\nu }{a}(I_{\infty }>b). 
\end{align}
By the definition of $\tpn $, the left-hand side of \eqref{;conv} 
is equal to 
\begin{align*}
  \frac{\ex{\nu }{a}[(R_{t_{n}})^{-2\nu };I_{t_{n}}>b]}
 {\ex{\nu }{a}[(R_{t_{n}})^{-2\nu }]}. 
\end{align*}
As the sequence $\{ t_{n}\} _{n\in \N }$ is arbitrarily taken, 
we now conclude that 
\begin{align}\label{;preq2}
 \lim _{t\to \infty } \frac{\ex{\nu }{a}[(R_{t})^{-2\nu };I_{t}>b]}
 {\ex{\nu }{a}[(R_{t})^{-2\nu }]}
 =\pr{\nu }{a}(I_{\infty }>b), 
\end{align}
which proves \thetag{i} by \lref{;keylem}, \eqref{;idistr} and 
\eqref{;iden2}. \\
\thetag{ii} By \pref{;absrel} we have 
\begin{align}\label{;preq3}
 \pr{-\nu }{a}(\tau _{b}>t)\equiv 
 \pr{-\nu }{a}(I_{t}>b)
 =\ex{\nu }{a}\left[ \left( \frac{a}{R_{t}}\right) ^{2\nu }; 
 I_{t}>b\right] . 
\end{align}
The assertion follows from this and \eqref{;preq2}. 
\end{proof}

The same reasoning as the proof of \tref{;th1}\,\thetag{i} also yields the 
\begin{prop}\label{;keyprop}
 For any continuous function $f:\R \to \R $, we have 
 \begin{align*}
  \lim _{t\to \infty }t^{\nu }
  \ex{\nu }{a}\left[ 
  f(I_{t})(R_{t})^{-2\nu }
  \right] 
  =\frac{2\nu }{2^{\nu }a^{2\nu }\Gamma (\nu +1)}
  \int _{0}^{a}z^{2\nu -1}f(z)\,dz. 
 \end{align*}
\end{prop}
\begin{proof}
 We keep the notation in the proof of \tref{;th1}\,\thetag{i}. Note that 
 the mapping 
 \begin{align*}
  \Om \ni \omega \mapsto \inf _{t\ge 0}\left\{ 
  a\wedge (\omega (t)\vee 0)
  \right\} =:I^{a,+}(\omega )
 \end{align*}
 is bounded and continuous, and that $I_{\infty }=I^{a,+}(R)$ 
 $\pr{\nu }{a}$-a.s. As $\{ \tpn \} _{n\in \N }$ converges 
 weakly to $\pr{\nu }{a}$, we have for any continuous function $f$ on $\R $, 
 \begin{align*}
  \int _{\Om }f\left( I^{a,+}(R)\right) d\tpn 
  \xrightarrow[n\to \infty ]{}\ex{\nu }{a}\left[ f(I_{\infty })\right] . 
 \end{align*}
 By the definition of $\tpn $, the left-hand side is equal to 
 \begin{align*}
  \frac{\ex{\nu }{a}\left[ f(I_{t_{n}})(R_{t_{n}})^{-2\nu }\right] }
  {\ex{\nu }{a}\left[ (R_{t_{n}})^{-2\nu }\right] }. 
 \end{align*}
 The rest of the proof proceeds in the same way as the proof of 
 \tref{;th1}\,\thetag{i}. 
\end{proof}

As an application of this proposition, we may prove further the following 
asymptotic formulae: We set 
\begin{align*}
 \ri :=\inf \{ t\ge 0;\,R_{t}=I_{\infty }\} ; 
\end{align*}
as we will see in \pref{;unique} below, $\ri $ is a.s.\ the unique 
time at which the global infimum $I_{\infty }$ is attained. 

\begin{thm}\label{;tcor}
 \thetag{i}~For any $0\le b\le a$, it holds that 
 \begin{align}\label{;tcor1}
  \lim _{t\to \infty }t^{\nu }
  \pr{\nu }{a}\left( I_{t}-I_{\infty }>b\right) 
  =\frac{2\nu }{2^{\nu }a^{2\nu }\Gamma (\nu +1)}
  \int _{b}^{a}z^{2\nu -1}(z-b)^{2\nu }\,dz. 
 \end{align}
 In particular, 
 \begin{align}\label{;tcor2}
  \lim _{t\to \infty }t^{\nu }\pr{\nu }{a}(\ri >t)
  =\frac{a^{2\nu }}{2^{\nu +1}\Gamma (\nu +1)}. 
 \end{align}
 \thetag{ii}~For any continuous function $g:\R \to \R $, it holds that 
 \begin{align*}
  \lim _{t\to \infty }t^{\nu }\left\{ 
  \ex{\nu }{a}[g(I_{\infty })]-\ex{\nu }{a}[g(I_{t})]
  \right\} 
  =\frac{2\nu }{2^{\nu }a^{2\nu }\Gamma (\nu +1)}
  \int _{0}^{a}(a^{2\nu }-2z^{2\nu })z^{2\nu -1}g(z)\,dz. 
 \end{align*}
\end{thm}

\begin{proof}
 \thetag{i} By the Markov property of Bessel process and by \eqref{;idistr}, 
 \begin{align}
  \pr{\nu }{a}(I_{t}-I_{\infty }>b)
  &=\ex{\nu }{a}\left[ 
  \pr{\nu }{x}(y-y\wedge I_{\infty }>b)\big| _{(x,y)=(R_{t},I_{t})}
  \right] \notag \\
  &=\ex{\nu }{a}\left[ 
  \frac{(I_{t}-b)^{2\nu }}{(R_{t})^{2\nu }};I_{t}>b
  \right] . \label{;cond1}
 \end{align}
 Taking $f(z)=\{ (z-b)\vee 0\} ^{2\nu }$ in \pref{;keyprop} 
 leads to \eqref{;tcor1}. The latter equality \eqref{;tcor2} follows by 
 taking $b=0$ in \eqref{;tcor1}; indeed, as seen in the proof 
 of \pref{;unique}, one has 
 $\pr{\nu }{a}(\ri >t)=\pr{\nu }{a}(I_{t}>I_{\infty })$. \\
 \thetag{ii} Again by the Markov property, 
 \begin{align*}
  \ex{\nu }{a}\left[ g(I_{\infty })| \calF _{t}\right] 
  =\ex{\nu }{x}\left[ g(y\wedge I_{\infty })\right] \!
  \big| _{(x,y)=(R_{t},I_{t})} \quad \text{$\pr{\nu }{a}$-a.s.}
 \end{align*}
 for every $t>0$. By \eqref{;idistr}, the $\pr{\nu }{x}$-expectation 
 on the right-hand side is calculated as 
 \begin{align*}
  g(y)+\frac{h(y)}{x^{2\nu }}, \quad 
  h(y):=2\nu \int _{0}^{y}z^{2\nu -1}g(z)\,dz-y^{2\nu }g(y). 
 \end{align*}
 Hence we have 
 \begin{align*}
  \ex{\nu }{a}[g(I_{\infty })]-\ex{\nu }{a}[g(I_{t})]
  =\ex{\nu }{a}\left[ 
  \frac{h(I_{t})}{(R_{t})^{2\nu }}
  \right] . 
 \end{align*}
 Taking $f=h$ in \pref{;keyprop} concludes the proof. 
\end{proof}

We give a remark on \tref{;tcor}\,\thetag{ii}. 
\begin{rem}
\thetag{1}~We may allow the function $g$ to have the set of 
discontinuity with Lebesgue measure $0$; in particular, taking 
$g=\ind _{(b,\infty )}$ recovers \tref{;th1}\,\thetag{i}. \\
\thetag{2}~For the function $h$ defined in the proof, the process 
 \begin{align*}
  g(I_{t})+\frac{h(I_{t})}{(R_{t})^{2\nu }}, \quad t\ge 0, 
 \end{align*}
 is, by definition, an $\{ \calF _{t}\} $-martingale 
 under $\pr{\nu }{a}$, which may be associated with the so-called 
 {\it Az\'ema-Yor martingales} (see \cite{ay}); 
 in fact, $\{ (R_{t})^{-2\nu };t\ge 0\} $ is an 
 $\{ \calF _{t}\} $-local martingale and 
 $\sup \limits_{0\le s\le t}(R_{s})^{-2\nu }=(I_{t})^{-2\nu }$. 
\end{rem}

We may also relate \eqref{;tcor2} to \tref{;th1}\,\thetag{ii} in the 
following manner: 

\begin{proof}[Proof of \eqref{;tcor2} via \tref{;th1}\,\thetag{ii}] 
 Note that by taking $b=0$ in \eqref{;cond1}, 
 \begin{align*}
  \pr{\nu }{a}(\ri >t)&=\pr{\nu }{a}(I_{t}>I_{\infty })\\
  &=\ex{\nu }{a}\left[ 
  \left( \frac{I_{t}}{R_{t}}\right) ^{2\nu }
  \right] . 
 \end{align*}
 By the absolute continuity relation \pref{;absrel} and 
 by Fubini's theorem, this is rewritten as 
 \begin{align*}
  \ex{-\nu }{a}\left[ \left( \frac{I_{t}}{a}\right) ^{2\nu }; 
  t<\tau _{0}\right] 
  &=\frac{2\nu }{a^{2\nu }}\int _{0}^{a}z^{2\nu -1}
  \pr{-\nu }{a}(I_{t}>z)\,dz\\
  &=\frac{2\nu }{a^{2\nu }}\int _{0}^{a}z^{2\nu -1}
  \pr{-\nu }{a}(\tau _{z}>t)\,dz. 
 \end{align*}
 For every $z\in (0,a)$ we have by \tref{;th1}\,\thetag{ii} and 
 \eqref{;preq3}, 
 \begin{align*}
  a^{2\nu }\ge \frac{\pr{-\nu }{a}(\tau _{z}>t)}
  {\ex{\nu }{a}[(R_{t})^{-2\nu }]}
  \xrightarrow[t\to \infty ]{}a^{2\nu }\left\{ 
  1-\left( \frac{z}{a}\right) ^{2\nu }
  \right\} , 
 \end{align*}
 and hence the bounded convergence theorem yields 
 \begin{align*}
  \lim _{t\to \infty }
  \frac{\pr{\nu }{a}(\ri >t)}
  {\ex{\nu }{a}[(R_{t})^{-2\nu }]}
  &=2\nu \int _{0}^{a}z^{2\nu -1}\left\{ 
  1-\left( \frac{z}{a}\right) ^{2\nu }
  \right\} dz\\
  &=\frac{1}{2}a^{2\nu }. 
 \end{align*}
 This shows \eqref{;tcor2} by \lref{;keylem}. 
\end{proof}
 
We conclude this section with a remark on the above proof. 
\begin{rem}\label{;decomp}
 In the proof we have just observed the identity 
 \begin{align}\label{;disint}
  \pr{\nu }{a}(\ri >t)=\frac{2\nu }{a^{2\nu }}
  \int _{0}^{a}z^{2\nu -1}\pr{-\nu }{a}(\tau _{z}>t)\,dz , 
 \end{align}
 which may easily be extended, thanks to the Markov property, to 
 \begin{align*}
  \pr{\nu }{a}\left( A\cap \{ \ri >t\} \right) =\frac{2\nu }{a^{2\nu }}
  \int _{0}^{a}z^{2\nu -1}
  \pr{-\nu }{a}\left( A\cap \{ \tau _{z}>t\} \right) dz 
 \end{align*}
 for any $A\in \calF _{t}$. This relation shows that 
 the process $\{ R_{t};0\le t\le \ri \} $ under $\pr{\nu }{a}$ 
 is identical in law with $\{ \xi _{t};0\le t\le \tau _{Z}(\xi )\} $, 
 where $\xi $ is a Bessel process with index $-\nu $ starting from $a$ 
 and $Z$ is a random variable independent of $\xi $ and distributed as 
 $(2\nu /a^{2\nu })z^{2\nu -1}\,dz,\,z\in (0,a)$. 
 In the case $\nu =1/2$, this partly recovers the path decomposition 
 of $3$-dimensional Bessel process due to D.~Williams 
 (e.g., \cite[Theorem~VI.3.11]{rey}). We also refer the reader to 
 \cite[Corollary~4.14]{fi} for identities as \eqref{;disint} 
 in a general framework of diffusion processes. 
\end{rem}

\ND\cnu{C_{\nu}}
\section{Asymptotic estimates for remainders}\label{;improved}
Independently of the argument used in the previous section, 
we prove in this section the next two theorems, which give 
sharp asymptotics for remainders in \eqref{;asym1} and \eqref{;asym2}. 
In the sequel we fix $0\le b<a$ and set 
\begin{align*}
 \cnu :=\frac{a^{2\nu }-b^{2\nu }}{2^{\nu }\Gamma (\nu +1)} 
\end{align*}
for every positive $\nu $. When $\nu <1$, we also set 
\begin{align}\label{;kpn}
 \kappa _{\nu }
 :=\int _{1}^{\infty }\frac{(v+1)^{2\nu }-v^{2\nu }}{v^{\nu +1}}\,dv 
 \in (0,\infty ). 
\end{align}

\begin{thm}\label{;improve}
It holds that 
\begin{align*}
 \thetag{i}&~\text{for $\nu <1$}, \quad \lim _{t\to \infty }t^{2\nu }
 \left( 
 \pr{\nu }{a}(\infty >\tau _{b}>t)-
 \frac{b^{2\nu }}{a^{2\nu }}\frac{\cnu }{t^{\nu }}
 \right) 
 =\frac{b^{4\nu }}{(2a^{2})^{\nu }\Gamma (\nu +1)}\cnu 
 \left( 
 1-\nu \kappa _{\nu }
 \right) ; \\
 \thetag{ii}&~\text{for $\nu =1$}, \quad 
 \lim _{t\to \infty }\frac{t^2}{\log t}
 \left( 
 \pr{1}{a}(\infty >\tau _{b}>t)-
 \frac{b^2}{a^2}\frac{C_{1}}{t}
 \right)  =-\frac{b^{4}}{a^2}C_{1}; \\
 \thetag{iii}&~\text{for $\nu >1$}, \\
 &\hspace{5.4em}-\infty <\liminf _{t\to \infty }t^{\nu +1}
 \left( 
 \pr{\nu }{a}(\infty >\tau _{b}>t)-
 \frac{b^{2\nu }}{a^{2\nu }}\frac{\cnu }{t^{\nu }}
 \right) \\
 &\hspace{7.6em}\le \limsup _{t\to \infty }t^{\nu +1}
 \left( 
 \pr{\nu }{a}(\infty >\tau _{b}>t)-
 \frac{b^{2\nu }}{a^{2\nu }}\frac{\cnu }{t^{\nu }}
 \right) <0. 
\end{align*}
\end{thm}

\begin{thm}\label{;improve2}
It holds that 
\begin{align*}
 \thetag{i}&~\text{for $\nu <1$}, \quad 
 \lim _{t\to \infty }t^{2\nu }
 \left( 
 \pr{-\nu }{a}(\tau _{b}>t)-\frac{\cnu }{t^{\nu }}
 \right) =
 \frac{b^{2\nu }}{2^{\nu }\Gamma (\nu +1)}\cnu 
 \left( 
 1-\nu \kappa _{\nu }
 \right) ; \\
 \thetag{ii}&~\text{for $\nu =1$}, \quad 
 \lim _{t\to \infty }\frac{t^2}{\log t}
 \left( 
 \pr{-1}{a}(\tau _{b}>t)-\frac{C_{1}}{t}
 \right) =-b^{2}C_{1}; \\
 \thetag{iii}&~\text{for $\nu >1$}, \\
 -&\infty <\liminf _{t\to \infty }t^{\nu +1}
 \left( 
 \pr{-\nu }{a}(\tau _{b}>t)-\frac{\cnu }{t^{\nu }}
 \right) 
 \le \limsup _{t\to \infty }t^{\nu +1}
 \left( 
 \pr{-\nu }{a}(\tau _{b}>t)-\frac{\cnu }{t^{\nu }}
 \right) <0. 
\end{align*}
\end{thm}

Notice that since 
\begin{align}\label{;relation}
 \pr{\nu }{a}(\infty >\tau _{b}>t)=
 \left( \frac{b}{a}\right) ^{2\nu }\pr{-\nu }{a}(\tau _{b}>t) 
\end{align} 
by \eqref{;iden2} and \eqref{;preq3}, it suffices to prove \tref{;improve2}. 
The proof utilizes the following relation for hitting distributions: 
\begin{lem}\label{;hitting}
 It holds that for every $t>0$, 
 \begin{align*}
  &\pr{-\nu }{a}(\tau _{b}>t) \notag \\
  &=\frac{1}{\pr{-\nu }{b}(\tau _{0}\le t)}
  \left\{  
  \pr{-\nu }{a}(\tau _{0}>t)-\pr{-\nu }{b}(\tau _{0}>t)
  -\int _{D_{t}}\pr{-\nu }{a}(\tau _{b}\in ds)\,
  \pr{-\nu }{b}(\tau _{0}\in du)
  \right\} , 
 \end{align*}
 where 
 \begin{align*}
  D_{t}:=\left\{ 
  (s,u)\in (0,\infty )^{2};\,
  s+u>t, s\le t,u\le t
  \right\} . 
 \end{align*}
\end{lem}

\begin{proof}
 By the strong Markov property, 
 \begin{align*}
  \pr{-\nu }{a}(\tau _{0}>t,\tau _{b}\le t)
  =\int _{0}^{t}\pr{-\nu }{a}(\tau _{b}\in ds)
  \pr{-\nu }{b}(\tau _{0}>t-s). 
 \end{align*}
 By the definition of $D_{t}$, we may rewrite the right-hand side as 
 \begin{align*}
  \int _{D_{t}}\pr{-\nu }{a}(\tau _{b}\in ds)
  \pr{-\nu }{b}(\tau _{0}\in du)+
  \pr{-\nu }{a}(\tau _{b}\le t)\pr{-\nu }{b}(\tau _{0}>t). 
 \end{align*}
 On the other hand, the left-hand side is equal to 
 \begin{align*}
  \pr{-\nu }{a}(\tau _{0}>t)-\pr{-\nu }{a}(\tau _{0}>t,\tau _{b}>t)
  =\pr{-\nu }{a}(\tau _{0}>t)-\pr{-\nu }{a}(\tau _{b}>t)
 \end{align*}
 since $\tau _{0}\ge \tau _{b}$ $\pr{-\nu }{a}$-a.s. 
 Combining these leads to the desired identity. 
\end{proof}

As a preparatory step to the proof of \tref{;improve2}, 
we give another proof of \tref{;th1}\,\thetag{ii} 
using \lref{;hitting}. Set 
\begin{align*}
 I(t):=\frac{\pr{-\nu }{a}(\tau _{0}>t)-\pr{-\nu }{b}(\tau _{0}>t)}
 {\pr{-\nu }{b}(\tau _{0}\le t)}, && 
 J(t):=\frac{\int _{D_{t}}\pr{-\nu }{a}(\tau _{b}\in ds)\,
  \pr{-\nu }{b}(\tau _{0}\in du)}{\pr{-\nu }{b}(\tau _{0}\le t)}
\end{align*}
so that 
\begin{align}\label{;iden4}
 \pr{-\nu }{a}(\tau _{b}>t)=I(t)-J(t) 
\end{align}
by \lref{;hitting}. Since we have the expression 
\begin{align}\label{;tail}
 \pr{-\nu }{x}(\tau _{0}>t) =
 \frac{x^{2\nu }}{2^{\nu }\Gamma (\nu )}
  \int _{t}^\infty \frac{ds}{s^{\nu +1}}
  \exp \left( -\frac{x^2}{2s}\right) , \quad t>0, 
\end{align}
for every $x\ge 0$ (see \rref{;gamma}), it is immediate that 
\begin{align*}
 \lim _{t\to \infty }t^{\nu }I(t)
 =\cnu . 
\end{align*}
Therefore in order to prove \tref{;th1}\,\thetag{ii}, it suffices to show 
that 
\begin{align*}
 \lim _{t\to \infty }t^{\nu }J(t)=0. 
\end{align*}
To this end, take $t,\la >0$ in such a way that $t>\la $. Because of the 
inclusion 
\begin{multline*}
 D_{t}\subset 
 \left\{ 
 (s,u)\in (0,\infty )^2;\,t+\la \ge s+u>t
 \right\} \\
 \cup 
 \left\{ 
 (s,u)\in (0,\infty )^{2};\,
 s+u> t+\la ,s\le t,u\le t
 \right\} , 
\end{multline*}
we have 
\begin{align}\label{;j}
 &\pr{-\nu }{b}(\tau _{0}\le t)J(t) \notag \\
 &\le \pr{-\nu }{a}(t+\la \ge \tau _{0}>t)
 +\int _{\la }^{t}\pr{-\nu }{b}(\tau _{0}\in du)\,
 \pr{-\nu }{a}(t\ge \tau _{b}>t+\la -u) \notag \\
 &=:J_{1}(t;\la )+J_{2}(t;\la ), 
\end{align}
where the expression of $J_{1}$ is due to the strong Markov property. 
By \eqref{;tail} we have 
\begin{align*}
 J_{1}(t;\la )&\le \frac{a^{2\nu }}{2^{\nu }\Gamma (\nu )}
 \int _{t }^{t+\la }\frac{ds}{s^{\nu +1}}\\
 &=\frac{a^{2\nu }}{2^{\nu }\Gamma (\nu +1)}
 \left( \frac{1}{t}\right) ^{\nu }\left\{ 
 1-\left( 
 \frac{t}{t+\la }
 \right) ^{\nu }
 \right\} . 
\end{align*}
Since there exists a positive constant $c$ such that 
$
 1-x^{\nu }\le c(1-x)
$ 
for all $0\le x\le 1$, we obtain an estimate 
\begin{align}\label{;j1}
 J_{1}(t;\la )\le c_{1}\frac{\la }{t^{\nu +1}}. 
\end{align}
Here as well as in what follows, every $c_{i}$ denotes a positive 
constant dependent only on $a,b$ and $\nu $. As for $J_{2}$, we use \eqref{;tail} to 
rewrite 
\begin{align}\label{;j21}
 J_{2}(t;\la )=
 \frac{b^{2\nu }}{2^{\nu }\Gamma (\nu )}
 \int _{\la }^{t}\frac{du}{(t+\la -u)^{\nu +1}}\,
 \exp \left\{ 
 -\frac{b^2}{2(t+\la -u)}
 \right\} \pr{-\nu }{a}(t\ge \tau _{b}>u). 
\end{align}
Since $\tau _{b}\le \tau _{0}$ $\pr{-\nu }{a}$-a.s., 
\begin{align*}
 \pr{-\nu }{a}(t\ge \tau _{b}>u)&\le \pr{-\nu }{a}(\tau _{0}>u)\\
 &\le \frac{a^{2\nu }}{2^{\nu }\Gamma (\nu +1)}\frac{1}{u^{\nu }}
\end{align*}
by \eqref{;tail}. We substitute this estimate into \eqref{;j21} to 
obtain a bound 
\begin{align}\label{;j22}
 J_{2}(t;\la )\le c_{2}\int _{\la }^{t}
 \frac{du}{u^{\nu }(t+\la -u)^{\nu +1}}. 
\end{align}
We now fix $\ve \in (0,1)$ arbitrarily and let $\la =\ve t$. Then by 
\eqref{;j1}, 
\begin{align*}
 \limsup _{t\to \infty }t^{\nu }J_{1}(t;\ve t)\le c_{1}\ve . 
\end{align*}
On the other hand, by \eqref{;j22}, 
\begin{align*}
 J_{2}(t;\ve t)\le \frac{c_{2}}{(\ve t)^{\nu +1}}
 \int _{\ve t}^{t}\frac{du}{u^{\nu }}, 
\end{align*}
whence 
\begin{align*}
 \lim _{t\to \infty }t^{\nu }J_{2}(t;\ve t)=0. 
\end{align*}
Combining these with \eqref{;j}, we have 
\begin{align*}
 \limsup _{t\to \infty }t^{\nu }J(t)\le c_{1}\ve . 
\end{align*}
This shows \tref{;th1}\,\thetag{ii} as $\ve $ is arbitrary. 

We proceed to the proof of \tref{;improve2}. 
We begin with the following lemma: 
\begin{lem}\label{;recursive}
 One has for every $x\ge 0$ and $t>0$, 
 \begin{align*}
  \pr{-\nu }{x}(\tau _{0}>t)
  =\frac{x^{2\nu }}{(2t)^{\nu }\Gamma (\nu +1)}
  \exp \left( 
  -\frac{x^2}{2t}
  \right) 
  +\pr{-\nu -1}{x}(\tau _{0}>t). 
 \end{align*}
\end{lem}

\begin{proof}
 By integration by parts, 
 \begin{align*}
  \int _{t}^{\infty }\frac{ds}{s^{\nu +1}}\,\exp \left( 
  -\frac{x^2}{2s}
  \right) 
  =\frac{1}{\nu t^{\nu }}\exp \left( 
  -\frac{x^2}{2t}
  \right) 
  +\frac{x^{2}}{2\nu }\int _{t}^{\infty }\frac{ds}{s^{\nu +2}}\,
  \exp \left( 
  -\frac{x^2}{2s}
  \right) . 
 \end{align*}
 Plugging this expression into \eqref{;tail}, we obtain the equality. 
\end{proof} 

Using this lemma, we divide $I(t)$ into three parts: 
\begin{align*}
 I(t)=I_{1}(t)+I_{2}(t)+I_{3}(t), 
\end{align*}
where we set 
\begin{equation}\label{;is0}
\begin{split}
 I_{1}(t)&=\frac{1}{(2t)^{\nu }\Gamma (\nu +1)}
 \left\{ 
 a^{2\nu }\exp \left( 
  -\frac{a^2}{2t}
  \right) -b^{2\nu }\exp \left( 
  -\frac{b^2}{2t}
  \right) 
 \right\} , \\
 I_{2}(t)&=\pr{-\nu -1}{a}(\tau _{0}>t)-\pr{-\nu -1}{b}(\tau _{0}>t), \\
 I_{3}(t)&=\frac{\pr{-\nu }{b}(\tau _{0}>t)}{\pr{-\nu }{b}(\tau _{0}\le t)}
 \left\{ 
 \pr{-\nu }{a}(\tau _{0}>t)-\pr{-\nu }{b}(\tau _{0}>t)
 \right\} . 
\end{split}
\end{equation}
Using the fact that $(1-e^{-x})/x\xrightarrow[x\to 0]{}1$ 
for $I_1$ and \eqref{;tail} for $I_{2}$ and $I_{3}$, 
we see that 
\begin{equation*}
 \begin{split}
 I_{1}(t)&=\frac{\cnu }{t^{\nu }}
 -\frac{a^{2\nu +2}-b^{2\nu +2}}{(2t)^{\nu +1}\Gamma (\nu +1)}
 +o(t^{-\nu -1}), \\
 I_{2}(t)&=\frac{C_{\nu +1}}{t^{\nu +1}}
 +o(t^{-\nu -1}), \\
 I_{3}(t)&=\frac{b^{2\nu }}{2^{\nu }\Gamma (\nu +1)}
 \cdot \frac{\cnu }{t^{2\nu }}
 +o(t^{-2\nu }). 
 \end{split}
\end{equation*}
We put together these asymptotics into a proposition. 

\begin{prop}\label{;iasympt}
 It holds that as $t\to \infty $, 
 \begin{align*}
  \thetag{i}&~\text{for $\nu <1$}, \quad 
  I(t)=\frac{\cnu }{t^{\nu }}+\frac{b^{2\nu }}{2^{\nu }\Gamma (\nu +1)}
  \cdot \frac{\cnu }{t^{2\nu }}+o(t^{-2\nu }); \\
 \thetag{ii}&~\text{for $\nu =1$}, \quad 
  I(t)=\frac{C_{1}}{t}-\frac{C_{1}^{2}}{2t^2}+o(t^{-2}); \\
  \thetag{iii}&~\text{for $\nu >1$}, \quad 
  I(t)=\frac{\cnu }{t^{\nu }}
  -\frac{\nu C_{\nu +1}}{t^{\nu +1}}+o(t^{-\nu -1}). 
 \end{align*}
\end{prop}

\begin{rem}
 It is easily deduced from \eqref{;is0} that the terms of 
 $o$-symbol in \thetag{i}, \thetag{ii} and \thetag{iii} can be 
 sharpened by 
 $O\left( 
  1/t^{(3\nu )\wedge (\nu +1)}
  \right) $, $O(1/t^{3})$ and 
  $O\left( 1/t^{(2\nu )\wedge (\nu +2)}\right) $, respectively.  
\end{rem}

As to $J(t)$ in the decomposition \eqref{;iden4}, we let $t,\la >0$ be 
such that $t>\la $ and set 
\begin{align*}
 K(t;\la ):=\int _{\la }^{t}\pr{-\nu }{b}(\tau _{0}\in du)
 \left( 
 I(t+\la -u)-I(t)
 \right) . 
\end{align*}
Recall \eqref{;j}. 

\begin{lem}\label{;jest}
The following estimates hold true: 
\begin{align}
 \pr{-\nu }{b}(\tau _{0}\le \la )J(t)
 &\le J_{1}(t;\la )+K(t;\la ), \label{;jupper}\\
 \pr{-\nu }{b}(\tau _{0}\le \la )J(t)
 &\ge K(t;\la )-\int _{\la }^{t}\pr{-\nu }{b}(\tau _{0}\in du)
 J(t+\la -u). \label{;jlower}
\end{align}
\end{lem}

\begin{proof}
 By the definition of $J_{2}$ and \eqref{;iden4}, 
 \begin{align*}
  J_{2}(t;\la )&=\int _{\la }^{t}\pr{-\nu }{b}(\tau _{0}\in du)
  \left( 
  I(t+\la -u)-J(t+\la -u)-I(t)+J(t)
  \right) \\
  &\le K(t;\la )+\left\{ 
  \pr{-\nu }{b}(\tau _{0}\le t)
  -\pr{-\nu }{b}(\tau _{0}\le \la )
  \right\} J(t), 
 \end{align*}
 where the inequality is due to the nonnegativity of $J(t+\la -u)$. 
 Plugging this estimate into \eqref{;j}, we obtain \eqref{;jupper}. 
 The lower estimate \eqref{;jlower} is proved similarly since 
 \begin{align*}
  \pr{-\nu }{b}(\tau _{0}\le t)J(t)\ge J_{2}(t;\la )
 \end{align*}
 by the definition of $J_{2}$. 
\end{proof}

Set 
\begin{align*}
 c_{3}:=\sup _{t>0}
 t^{(2\nu )\wedge (\nu +1)}\left| 
 I(t)-\frac{\cnu }{t^{\nu }}
 \right| , 
\end{align*}
which is finite by \pref{;iasympt}. Then, by noting 
\begin{align*}
 \frac{\cnu }{t^{\nu }}-\frac{c_{3}}{t^{(2\nu )\wedge (\nu +1)}}
 \le I(t)\le \frac{\cnu }{t^{\nu }}+\frac{c_{3}}{t^{(2\nu )\wedge (\nu +1)}} 
\end{align*}
for all $t>0$ and by \eqref{;tail}, we have upper and lower bounds on 
$K(t;\la )$ as follows: 
\begin{align}
 K(t;\la )&\le \frac{b^{2\nu }}{2^{\nu }\Gamma (\nu )}\cnu 
 K_{1}(t;\la )+\frac{b^{2\nu }}{2^{\nu }\Gamma (\nu )}c_{3}K_{2}(t;\la )
 +\frac{c_{3}\pr{-\nu }{b}(t\ge \tau _{0}>\la )}{t^{(2\nu )\wedge (\nu +1)}}, 
 \label{;kupper} \\
 K(t;\la )&\ge \frac{b^{2\nu }}{2^{\nu }\Gamma (\nu )}\cnu 
 \exp \left( -\frac{b^2}{2\la }\right) K_{1}(t;\la )
 -\frac{b^{2\nu }}{2^{\nu }\Gamma (\nu )}c_{3}K_{2}(t;\la )
 -\frac{c_{3}\pr{-\nu }{b}(t\ge \tau _{0}>\la )}{t^{(2\nu )\wedge (\nu +1)}}, 
 \label{;klower} 
\end{align}
where 
\begin{align*}
 K_{1}(t;\la ):=\int _{\la }^{t}\frac{du}{u^{\nu +1}}
 \left\{ 
 \frac{1}{(t+\la -u)^{\nu }}-\frac{1}{t^{\nu }}
 \right\} , 
 && 
 K_{2}(t;\la ):=\int _{\la }^{t}
 \frac{du}{u^{\nu +1}(t+\la -u)^{(2\nu )\wedge (\nu +1)}}. 
\end{align*}

\begin{lem}\label{;kest}
 \thetag{1}~It holds that as $t\to \infty $, 
 \begin{align*}
  K_{1}(t;\la )=\frac{1}{(t+\la )^{2\nu }}
  \int _{1}^{t/\la }\frac{(v+1)^{2\nu }-v^{2\nu }}{v^{\nu +1}}\,dv
  +O(t^{-\nu -1}). 
 \end{align*}
 \thetag{2}~It holds that 
 \begin{align*}
  \limsup _{t\to \infty }t^{(2\nu )\wedge (\nu +1)}
  K_{2}(t;\la )\le \frac{2}{\nu \la ^{\nu }}. 
 \end{align*}
\end{lem}

\begin{proof}
\thetag{1}~By \lref{;integral} in Appendix, 
\begin{align*}
 K_{1}(t;\la )&=\frac{1}{\la ^{\nu }(t+\la )^{2\nu }}
 \int _{\la }^{t}\frac{(u+\la )^{2\nu }}{u^{\nu +1}}\,du
 -\frac{1}{\nu }\left( 
 \frac{1}{\la ^{\nu }}-\frac{1}{t^{\nu }}
 \right) \frac{1}{t^{\nu }}\\
 &=\frac{1}{\la ^{\nu }(t+\la )^{2\nu }}
 \int _{\la }^{t}\frac{(u+\la )^{2\nu }-u^{2\nu }}{u^{\nu +1}}\,du
 -\frac{t^{\nu }-\la ^{\nu }}{\nu \la ^{\nu }t^{2\nu }}
 \left\{ 
 1-\left( \frac{t}{t+\la }\right) ^{2\nu }
 \right\} . 
\end{align*}
Since the second term in the last member is of order $O(t^{-\nu -1})$, 
the assertion follows by changing variables with $u=\la v$ in the 
integral of the first term. \\
\thetag{2}~When $\nu <1$, we have by \lref{;integral}, 
\begin{align*}
 K_{2}(t;\la )&=\left\{ 
 \frac{1}{\la (t+\la )}
 \right\} ^{3\nu }\int _{\la }^{t}(u+\la )^{3\nu -1}
 \left( 
 \frac{\la ^{\nu +1}}{u^{\nu +1}}+\frac{\la ^{2\nu }}{u^{2\nu }}
 \right) du\\
 &\le \frac{2}{\la ^{\nu }(t+\la )^{3\nu }}
 \int _{\la }^{t}\frac{(u+\la )^{3\nu -1}}{u^{2\nu }}\,du, 
\end{align*}
from which the assertion follows readily. Here for the inequality, 
we used the fact that $(\la /u)^{\nu +1}\le (\la /u)^{2\nu }$ 
for $u\ge \la $ as $\nu <1$. The case $\nu \ge 1$ can be proved 
similarly (in fact, the limit exists in both cases).  
\end{proof}

We are in a position to prove \tref{;improve2}. 
\begin{proof}[Proof of \tref{;improve2}] 
 \thetag{i}~In view of the decomposition \eqref{;iden4} and 
 \pref{;iasympt}\,\thetag{i}, what to show is that 
 \begin{align}\label{;target1}
  \lim _{t\to \infty }t^{2\nu }J(t)=
  \frac{b^{2\nu }}{2^{\nu }\Gamma (\nu )}\cnu 
  \kappa _{\nu }. 
 \end{align}
 Fix $\la >0$ arbitrarily. By \eqref{;kupper}, \lref{;kest} and 
 the definition \eqref{;kpn} of $\kappa _{\nu }$, we have 
 \begin{align*}
  \limsup _{t\to \infty }t^{2\nu }K(t;\la )
  \le \frac{b^{2\nu }}{2^{\nu }\Gamma (\nu )}\cnu 
  \kappa _{\nu }
  +\frac{b^{2\nu }}{2^{\nu }\Gamma (\nu )}c_{3}\cdot \frac{2}{\nu \la ^{\nu }}
  +c_{3}\pr{-\nu }{b}(\tau _{0}>\la ). 
 \end{align*}
 By \eqref{;jupper} and \eqref{;j1}, we see that the above estimate is also 
 valid with the left-hand side replaced by 
 \begin{align*}
  \pr{-\nu }{b}(\tau _{0}\le \la )\cdot \limsup _{t\to \infty }t^{2\nu }J(t). 
 \end{align*}
 As $\la $ is arbitrary, we obtain by letting $\la \to \infty $,  
 \begin{align}\label{;jls1}
  \limsup _{t\to \infty }t^{2\nu }J(t)\le 
  \frac{b^{2\nu }}{2^{\nu }\Gamma (\nu )}\cnu 
  \kappa _{\nu }. 
 \end{align}
 We may use this upper bound to estimate the second term on the 
 right-hand side of \eqref{;jlower} in such a way that 
 \begin{align*}
  \int _{\la }^{t}\pr{-\nu }{b}(\tau _{0}\in du)J(t+\la -u)
  \le c_{4}K_{2}(t;\la )
 \end{align*} 
 for some $c_{4}$. Then by \eqref{;klower} and \lref{;kest}, 
 \begin{align*}
  &\pr{-\nu }{b}(\tau _{0}\le \la )\cdot \liminf _{t\to \infty }t^{2\nu }J(t)\\
  &\ge \frac{b^{2\nu }}{2^{\nu }\Gamma (\nu )}\cnu 
  \exp \left( -\frac{b^{2}}{2\la }\right) \kappa _{\nu }
  -\left( \frac{b^{2\nu }}{2^{\nu }\Gamma (\nu )}c_{3}+c_{4}\right) 
  \frac{2}{\nu \la ^{\nu }}-c_{3}\pr{-\nu }{b}(\tau _{0}>\la ), 
 \end{align*}
 and hence letting $\la \to \infty $ also yields 
 \begin{align*}
  \liminf _{t\to \infty }t^{2\nu }J(t)\ge 
  \frac{b^{2\nu }}{2^{\nu }\Gamma (\nu )}\cnu 
  \kappa _{\nu }. 
 \end{align*}
 This together with \eqref{;jls1}, proves \eqref{;target1}. \\
 \thetag{ii}~We show that 
 \begin{align}\label{;target2}
  \lim _{t\to \infty }\frac{t^2}{\log t}J(t)=b^2C_{1}. 
 \end{align}By \lref{;kest}\,\thetag{1}, 
 \begin{align}
  K_{1}(t;\la )=&\frac{1}{(t+\la )^2}\left( 
  2\log \frac{t}{\la }-\frac{\la }{t}+1
  \right) +O(t^{-2}), \notag 
 \intertext{hence for any $\la >0$,}
  &\quad \lim _{t\to \infty }\frac{t^{2}}{\log t}K_{1}(t;\la )=2. 
  \label{;klim2}
 \end{align}
 Therefore by \eqref{;kupper} and \lref{;kest}\,\thetag{2}, 
 \begin{align*}
  \limsup _{t\to \infty }\frac{t^2}{\log t}K(t;\la )
  \le \frac{b^2}{2}C_{1}\times 2=b^{2}C_{1}, 
 \end{align*}
 which entails that by \eqref{;jupper}, \eqref{;j1} and arbitrariness of $\la $, 
 \begin{align}\label{;jls2}
  \limsup _{t\to \infty }\frac{t^{2}}{\log t}J(t)\le b^{2}C_{1}. 
 \end{align}
 By this estimate, we bound the second term on the right-hand side 
 of \eqref{;jlower} as 
 \begin{align*}
  \int _{\la }^{t}\pr{-1}{b}(\tau _{0}\in du)J(t+\la -u)
  &\le c_{5}\int _{\la }^{t}\frac{du}{u^{2}}\cdot 
  \frac{\log (t+\la -u)}{(t+\la -u)^2}\\
  &\le c_{5}\log t\cdot K_{2}(t;\la ). 
 \end{align*}
 Combining this with \lref{;kest}\,\thetag{2} and \eqref{;klim2}, 
 we see that 
 \begin{align*}
  \pr{-1}{b}(\tau _{0}\le \la )\cdot 
  \liminf _{t\to \infty }\frac{t^2}{\log t}J(t)
  \ge b^2C_{1}\exp \left( -\frac{b^2}{2\la }\right) -\frac{2c_{5}}{\la }, 
 \end{align*}
 and hence that 
 \begin{align}\label{;jli2}
  \liminf _{t\to \infty }\frac{t^2}{\log t}J(t)\ge b^2C_{1}. 
 \end{align}
 By \eqref{;jls2} and \eqref{;jli2}, we conclude \eqref{;target2}. 
 
 \noindent 
 \thetag{iii}~It suffices to prove 
 \begin{align}\label{;target3}
  \limsup _{t\to \infty }t^{\nu +1}J(t)<\infty  
 \end{align}
 by \eqref{;iden4} and \pref{;iasympt}\,\thetag{iii}. 
 For each fixed $\la >0$, $K_{1}(t;\la )=O(t^{-\nu -1})$ by 
 \lref{;kest}\,\thetag{1}. Therefore by \lref{;kest}\,\thetag{2} and 
 \eqref{;kupper}, 
 \begin{align*}
  \limsup _{t\to \infty }t^{\nu +1}K(t;\la )<\infty . 
 \end{align*}
 Combining this with \eqref{;jupper} and \eqref{;j1} leads to 
 \begin{align*}
  \pr{-\nu }{b}(\tau _{0}\le \la )\cdot 
  \limsup _{t\to \infty }t^{\nu +1}J(t)<\infty , 
 \end{align*}
 and hence \eqref{;target3}. The proof is complete. 
\end{proof}

\begin{proof}[Proof of \tref{;improve}]
 \tref{;improve2} combined with \eqref{;relation} 
 shows the theorem. 
\end{proof}

We close this section with a remark on \tref{;improve}. 
\begin{rem}\label{;rimprove}
\thetag{1}~In the case $\nu =1/2$, namely the case that the dimension 
$\delta =2(\nu +1)$ is $3$, the limit exhibited in 
\tref{;improve}\,\thetag{i} is equal to $0$ since 
\begin{align*}
 1-\nu \kappa _{\nu }
 =1-\frac{1}{2}\int _{1}^{\infty }\frac{dv}{v^{3/2}}=0, 
\end{align*}
which is consistent with the fact that 
\begin{align*}
 \pr{1/2}{a}(\infty >\tau _{b}>t)
 &=\frac{b}{a}\int _{t}^{\infty }
 \frac{a-b}{\sqrt{2\pi s^3}}
 \exp \left\{ 
 -\frac{(a-b)^{2}}{2s}
 \right\} ds\\
 &=\frac{b}{a}\cdot \frac{C_{1/2}}{t^{1/2}}+O(t^{-3/2}). 
\end{align*}
Such cancellation in asymptotic expansions resulting in the remainder 
$O(t^{-3/2})$ in dimension $3$, is observed in a generality by 
\cite[Proposition~2]{vdb}, where obtained are the asymptotic formulae for 
tail probabilities of hitting times of Brownian motion to general nonpolar 
compact sets in dimension greater than or equal to $3$. 

\noindent 
\thetag{2}~On the other hand, when $\nu \neq 1/2$, we 
observe that 
\begin{align}\label{;nocancel}
 1-\nu \kappa _{\nu }>0 \quad \text{for}\quad \nu <1/2 && 
 \text{and} && 1-\nu \kappa _{\nu }<0 \quad \text{for}\quad \nu >1/2, 
\end{align}
that is, the cancellation as in the case $\nu =1/2$ does not take place. 
To verify \eqref{;nocancel}, 
we change variables with $v=x/(1-x)$ in the definition \eqref{;kpn} 
of $\kappa _{\nu }$ to rewrite 
\begin{align*}
 \kappa _{\nu }
 =\int _{1/2}^{1}\frac{1-x^{2\nu }}{x^{\nu +1}(1-x)^{\nu +1}}\,dx, 
\end{align*}
which entails that $\kappa _{\nu }$ is (strictly) increasing 
in $\nu $ since for every $x\in (0,1)$, both 
$1-x^{2\nu }$ and $x^{-\nu }(1-x)^{-\nu }$ are increasing in $\nu $. 

\noindent 
\thetag{3}~In the case that $\delta $ is an integer 
greater than or equal to $4$, the assertions \thetag{ii} and \thetag{iii} of 
\tref{;improve} also agree with \cite[Proposition~2]{vdb}, 
the starting point $x$ of Brownian motion and a compact 
set $K$ therein being taken respectively as $|x|=a$ and 
$K$ the ball of radius $b$ centered at the origin. 
We remark that in the case $\delta =4$ (i.e., $\nu =1$), 
that proposition  
further reveals that the remainder after the term of order $(\log t)/t^{2}$ 
is $O(t^{-2})$; it also indicates that the constant $b^{4}$ appearing in 
the limit in \tref{;improve}\,\thetag{ii} arises from the square of the 
Newtonian capacity $8\pi ^{2}b^{2}$ of a ball of radius $b$ in $\R ^{4}$.  

\noindent 
\thetag{4}~At least for the decay rates of remainders, namely 
$O(t^{-2\nu })$ for $\nu <1$ ($\nu \neq 1/2$), $O((\log t)/t^{2})$ 
for $\nu =1$, and $O(t^{-\nu -1})$ for $\nu >1$ and $\nu =1/2$, 
they may also be deduced from a recent result by Uchiyama \cite{uch} 
that gives asymptotic estimates of the density function of 
$\tau _{b}$ which are valid uniformly in starting points $a$ including the 
case $\nu =0$ as well. 

\noindent 
\thetag{5}~By using an explicit representation for the 
distribution function of $\tau _{b}$, it is shown in 
\cite[Theorem~4.1\,\thetag{3}]{hm12} that in the case 
$\delta $ is an odd integer with $\delta \ge 3$, 
the remainder decays at rate $t^{-\nu -1}$. 

\noindent 
\thetag{6}~Although we do not give details here, we may also prove that 
in the case $\nu >1$, 
\begin{align*}
 \liminf _{t\to \infty }t^{\nu +1}J(t)
 \ge \frac{b^{2\nu }}{2^{\nu }\Gamma (\nu )}\cdot 
 \frac{a^{2}-b^{2}}{2(\nu -1)}, 
\end{align*}
the second factor on the right-hand side being the expected value 
of $\tau _{b}$ under $\pr{-\nu }{a}$. By this estimate, 
\pref{;iasympt}\,\thetag{iii} and \eqref{;relation}, 
the upper limit in \tref{;improve}\,\thetag{iii} is estimated from 
above by 
\begin{align*}
 -\left( 
 \frac{b}{a}
 \right) ^{2\nu }\!
 \left\{ 
 \nu C_{\nu +1}+\frac{b^{2\nu }(a^{2}-b^{2})}{2^{\nu +1}(\nu -1)\Gamma (\nu )}
 \right\} . 
\end{align*}

\end{rem}

\appendix 
\section*{Appendix}
\renewcommand{\thesection}{A}
\setcounter{equation}{0}
\setcounter{prop}{0}
\setcounter{lem}{0}
\setcounter{rem}{0}

We append proofs of auxiliary facts referred to in preceding 
sections. 

\subsection{Proof of \pref{;absrel}}

The relation \eqref{;opposite} may be deduced from the fact 
that the infinitesimal generator of Bessel process with positive index $\nu $ 
is identical with that of Bessel process with the opposite index, 
$h$-transformed by the function $h(x)=x^{2\nu },\,x>0$, that is harmonic in 
the sense that 
\begin{align*}
 \frac{1}{2}h''(x)+\frac{-2\nu +1}{2x}h'(x)=0. 
\end{align*}
For the reader's convenience, we give a proof of the proposition 
by means of a time-change and the Cameron-Martin relation. 
We refer the reader to 
\cite{yor} for the absolute continuity relationship for Bessel processes 
with nonnegative indices, which can also be proved by the same argument 
as below. 

\begin{proof}[Proof of \pref{;absrel}]
 Fix $a>0$ and set $b=\log a$. Let $B=\{ B_{t};t\ge 0\} $ be 
 a one-dimensional Brownian motion starting from $b$. For each 
 $\mu \in \R $, we denote by $\br{\mu }$ the Brownian motion 
 with drift $\mu $: $\br{\mu }_{t}=B_{t}+\mu t,\,t\ge 0$. 
 Let $X$ denote the coordinate process on $\Om =C([0,\infty );\R )$. 
 We define two functionals $A,\,\a $ of $X$ by 
 \begin{align*}
  A_{t}(X):=\int _{0}^{t}e^{2X_{s}}\,ds, \quad 
  \a _{t}(X):=\inf \{ s\ge 0;A_{s}(X)>t\} , \quad t\ge 0, 
 \end{align*}
 where we set $\inf \emptyset =\infty $. By Lamperti's relation 
 (see, e.g., \cite[Section~3]{my}), there exists a Bessel process $\be{\mu }$ 
 with index $\mu $ starting from $a$ such that 
 \begin{align}
  \exp \br{\mu }_{t}&=\be{\mu }_{A_{t}(\br{\mu })}, \quad t\ge 0, 
  \label{;Lamperti}
 \intertext{and hence by the definition of $\a $, }
  \exp \br{\mu }_{\a _{t}(\br{\mu })}&=\be{\mu }_{t}, 
  \quad t<\tau _{0}(\be{\mu }). \label{;Lamperti2}
 \end{align}
 Recall that $\tau _{0}(\be{\mu })=\infty $ a.s.\ for $\mu \ge 0$ 
 while $\tau _{0}(\be{\mu })<\infty $ a.s.\ for $\mu <0$; in fact, 
 \begin{align}\label{;aeq0}
  \tau _{0}(\be{\mu })=A_{\infty }(\br{\mu }) 
 \end{align}
 by \eqref{;Lamperti}. 
 We now fix $t>0$ and take 
 $\Gamma \in \calF _{t}=\sigma (X_{s},s\le t)$. Let $\nu >0$. 
 Then by \eqref{;Lamperti2}, 
 \begin{align}\label{;aeq1}
  P\left( \be{\nu }\in \Gamma \right) =
  P\left( 
  \exp \br{\nu }_{\a _{\cdot }(\br{\nu })}\in \Gamma 
  \right) . 
 \end{align}
 By definition, $\a _{t}(X)$ is a stopping time for the coordinate 
 process $X$. Therefore the Cameron-Martin formula entails that \eqref{;aeq1} 
 is equal to 
 \begin{align*}
  &e^{-2\nu b}E\left[ 
  \exp \left\{ 
  2\nu \br{-\nu }_{\a _{t}(\br{-\nu })}
  \right\} ; 
  \exp \br{-\nu }_{\a _{\cdot }(\br{-\nu })}\in \Gamma ,
  \a _{t}(\br{-\nu })<\infty 
  \right] \\
  &=a^{-2\nu }E\left[ 
  \bigl( \be{-\nu }_{t}\bigr) ^{2\nu }; 
  \be{-\nu }\in \Gamma ,\tau _{0}(\be{-\nu })>t
  \right] . 
 \end{align*}
 Here for the second line, we used \eqref{;Lamperti2}, and the 
 equivalence between $\a _{t}(\br{-\nu })<\infty $ and 
 $\tau _{0}(\be{-\nu })>t$ that  follows from \eqref{;aeq0}. 
 The proof is complete. 
\end{proof}

\begin{rem}\label{;gamma}
 By \eqref{;aeq0} and Dufresne's identity (see, e.g., \cite[Section~2]{my}), 
 it holds that under $\pr{-\nu }{a}$, 
 \begin{align*}
  \tau _{0}(R)\stackrel{(d)}{=}\frac{a^{2}}{2\gamma _{\nu }}. 
 \end{align*}
 Here $\gamma _{\nu }$ is a gamma random variable with parameter $\nu $. 
 Therefore one may find that 
 \begin{align*}
  \ex{\nu }{a}\left[ 
  \left( \frac{a}{R_{t}}\right) ^{2\nu }
  \right] &=\pr{-\nu }{a}(\tau _{0}>t) \\
  &=P\bigl( \gamma _{\nu }<a^{2}/(2t)\bigr) \\
  &=\frac{a^{2\nu }}{2^{\nu }\Gamma (\nu )}
  \int _{t}^\infty \frac{ds}{s^{\nu +1}}
  \exp \left( -\frac{a^2}{2s}\right) , 
 \end{align*}
 where the first equality follows from \eqref{;opposite}. 
\end{rem}

\subsection{Proof of \pref{;taboo}}
The proposition asserts that Bessel process with a negative index 
conditioned to stay positive is nothing but Bessel process with the opposite 
index. This seems to be a well-known fact and to have been rediscovered by 
several authors, see e.g., \cite[Section~7]{sh}; we also refer to \cite{cmm} 
for the case of drifted Brownian motions with nonsingular drift coefficients. 
The case $\nu =1/2$ goes back to Knight \cite[Theorem~3.1]{kn}. Roynette, 
Yor et al.\ extensively studied limit laws of Brownian motion normalized by 
various kinds of weight processes other than $\ind _{\{ \tau _{0}>t\} }$, 
referring to those studies as {\it penalisation problems}; see \cite{roy} 
and references therein, where usage of Scheff\'e's lemma we employ 
in the proof below is also found. For related studies concerning 
quasi-stationary distributions (Yaglom limits), refer to \cite{mv}. 

\begin{proof}[Proof of \pref{;taboo}]
Fix arbitrarily a sequence $\{ t_{n}\} _{n\in \N }$ of positive reals 
such that $\lim \limits_{n\to \infty }t_{n}=\infty $. Let 
$N\in \N $ be such that $t_{n}>t$ for all $n\ge N$. By \pref{;absrel} and 
the Markov property, we have for every $n\ge N$, 
\begin{align}
 \pr{-\nu }{a}(A\,|\,\tau _{0}>t_{n})
 &=\frac{\ex{\nu }{a}[(R_{t_{n}})^{-2\nu };A]}
 {\ex{\nu }{a}[(R_{t_{n}})^{-2\nu }]} \notag \\
 &=\ex{\nu }{a}\left[ 
 M_{n};A
 \right] , \label{;preq1d}
\end{align}
where we set 
\begin{align*}
 M_{n}=\frac{\ex{\nu }{R_{t}}[(R_{t_{n}-t})^{-2\nu }]}
 {\ex{\nu }{a}[(R_{t_{n}})^{-2\nu }]}. 
\end{align*}
By \lref{;keylem}, $M_{n}\to 1$ a.s.\ as $n\to \infty $. Moreover, 
$M_{n}\ge 0$ a.s.\ and 
$\ex{\nu }{a}[M_{n}]=1$ for all $n\ge N$. Hence by 
Scheff\'e's lemma, $\ex{\nu }{a}[|M_{n}-1|]\xrightarrow[n\to \infty ]{}0$, 
which entails that \eqref{;preq1d} converges to $\pr{\nu }{a}(A)$ 
as $n \to \infty $. Therefore we arrive at the conclusion as 
$\{ t_{n}\} _{n\in \N }$ is arbitrary. 
\end{proof}

\subsection{Statements and proofs of \pref{;unique} and \lref{;integral}}
\begin{prop}\label{;unique}
 Let $\nu >0$ and $a>0$. Under $\pr{\nu }{a}$, the time at which 
 the Bessel process $R$ attains its global infimum $I_\infty $ is 
 a.s.\ unique. 
\end{prop}

\def\T {\mathcal{T}}
\def\rb {\bar{\rho }_{\infty }}
\begin{proof}
 Write $\T \equiv \T (R)=\{ s\ge 0;R_{s}=I_{\infty }\} $. 
 Set $\ri =\inf \T $ as in \sref{;results} and $\rb =\sup \T $. Note that 
 $\T $ is compact a.s.\ since $R$ is continuous and 
 $\lim \limits_{s\to \infty }R_{s}=\infty $ a.s. It then holds that 
 \begin{align}\label{;a2-1}
  \{ \ri >t\} =\{ R_{s}>I_{\infty }\text{ for all }s\in [0,t]\} 
  \quad \text{a.s., }
 \end{align}
 namely the indicator functions of these two events are 
 equal a.s. Indeed, it is obvious that the left-hand event is included in 
 the right-hand event; for converse inclusion, since $t\notin \T $ and 
 $\T $ is compact a.s., we have $t<\inf \T =\ri $ a.s. By continuity, 
 the right-hand side of \eqref{;a2-1} is written as 
 $\bigl\{ \inf \limits_{0\le s\le t}R_{s}>I_{\infty }\bigr\} $, and hence 
 we have 
 \begin{align*}
  \{ \ri >t\} =\bigl\{ I_{t}>\inf _{s\ge t}R_{s}\bigr\} \quad \text{a.s.}
 \end{align*}
 Therefore by the Markov property and \eqref{;idistr}, 
 \begin{align*}
  \pr{\nu }{a}(\ri >t)=\ex{\nu }{a}\left[ 
  \left( \frac{I_{t}}{R_{t}}\right) ^{2\nu }
  \right] . 
 \end{align*}
 Similarly 
 \begin{align*}
  \{ \rb <t\} =\bigl\{ I_{t}<\inf _{s\ge t}R_{s}\bigr\} \quad \text{a.s., }
 \end{align*}
 from which it also follows that 
 \begin{align*}
  \pr{\nu }{a}(\rb \ge t)&=\pr{\nu }{a}
  \bigl( I_{t}\ge \inf _{s\ge t}R_{s}\bigr) \\
  &=\ex{\nu }{a}\left[ 
  \left( \frac{I_{t}}{R_{t}}\right) ^{2\nu }
  \right] . 
 \end{align*}
 By the dominated convergence theorem, the mapping 
 $[0,\infty )\ni t\mapsto \ex{\nu }{a}[(I_{t}/R_{t})^{2\nu }]$ 
 is continuous. Combining these we see that 
 $\ri $ and $\rb $ have the same distribution, which implies that 
 $\ri =\rb $ a.s.\ since $\ri \le \rb $. This ends the proof. 
\end{proof}

\begin{lem}\label{;integral}
 Let $c>0$ and $\al ,\beta \in \R $. For all $x\ge c$ one has 
 \begin{align*}
  \int _{c}^{x}\frac{dy}{y^{\al }(x+c-y)^{\beta }}
  =\left\{ \frac{1}{c(x+c)}\right\} ^{\al +\beta -1}
  \!\!\!\int _{c}^{x}(y+c)^{\al +\beta -2}
  \left( 
  \frac{c^{\al }}{y^{\al }}+\frac{c^{\beta }}{y^{\beta }}
  \right) dy. 
 \end{align*}
\end{lem}

\begin{proof}
 By changing variables with $y=x+c-z$, we see that the left-hand side of 
 the claimed identity is symmetric with respect to $\al $ and $\beta $, namely 
 \begin{align*}
  \int _{c}^{x}\frac{dz}{z^{\beta }(x+c-z)^{\al }}
  =\int _{c}^{x}\frac{dz}{z^{\al }(x+c-z)^{\beta }}, 
 \end{align*}
 which entails that it is equal to 
 \begin{align*}
  \int _{c}^{(x+c)/2}\left\{ 
  \frac{1}{z^{\al }(x+c-z)^{\beta }}+\frac{1}{z^{\beta }(x+c-z)^{\al }}
  \right\} dz. 
 \end{align*}
 Changing variables with $z=c(x+c)/(y+c)$ leads to the conclusion. 
\end{proof}
\medskip 

\noindent 
{\bf Acknowledgements.} The author would like to thank an anonymous 
referee for valuable comments, which helped to improve considerably the paper. 


\end{document}